\documentclass[11pt]{amsart}
 \usepackage{amsmath,amssymb,enumerate}
 \usepackage[all]{xy}
 \newtheorem{prop}{Proposition}[section]

 \newtheorem{lem}[prop]{Lemma}

 \newtheorem{thm}[prop]{Theorem}

 \newcommand{\B}{\mathbb B}
 
 \newcommand{\R}{\mathbb R}
 
 \newcommand{\C}{\mathbb C}

 \newcommand{\e}{\varepsilon}
 \newcommand{\f}{\varphi}
 
 \newcommand{\p}{\psi}

 \numberwithin{equation}{section}
 \newtheorem*{ackn}{Acknowledgement}

\begin{document}

  \title[Viscosity solutions to complex Monge-Amp\`ere equations]{Erratum to Viscosity solutions to complex Monge-Amp\`ere equations}

\setcounter{tocdepth}{1}

  \author{ Philippe Eyssidieux, Vincent Guedj, Ahmed Zeriahi} 

\address{Universit\'e Joseph Fourier et Institut Universitaire de France}

\email{Philippe.Eyssidieux@ujf-grenoble.fr }

\address{Institut de Math\'ematiques de Toulouse et Institut Universitaire de France}

\email{vincent.guedj@math.univ-toulouse.fr}

\address{Institut de Math\'ematiques de Toulouse, France}

\email{zeriahi@math.univ-toulouse.fr}

 \date{\today}

\begin{abstract}  
The proof of the comparison principle
in \cite{EGZ11} is not complete. We provide  here an alternative proof, valid
 in the ample locus of any big cohomology class,
 and discuss the  resulting modifications.
\end{abstract}

 \maketitle

 \section*{Introduction}

 Jeff Streets has informed the authors that the proof of \cite[Theorem 2.14]{EGZ11} is not correct as it stands: the localization procedure
 that we use does not provide enough information to control the Hessian of the penalization function $\f_3$ along the diagonal.
 
 We provide here a different approach which yields an alternative proof of the global comparison principle in the ample locus.
 This is sufficient for constructing unique viscosity solutions to degenerate complex Monge-Amp\`ere equations.
 The latter are continuous in the ample locus, but the continuity at the boundary remains an   open question.

 \begin{ackn} 
 We would like to thank Jeff Streets for pointing out the problem in the original proof
 and interesting exchanges.
\end{ackn}

 \section{Comparison principle in the ample locus} 
 
 Wer first recall the context.
   Let $X$ be a compact K\"ahler manifold of dimension $n$.  
We consider the complex Monge-Amp\`ere equation on $X$,
\begin{equation} \label{eq:CMA}
 e^{\f } \mu - (\omega +dd^c \f)^n = 0,
 \end{equation}
where 
  \begin{itemize}
  \item $\omega $ is a continuous  closed    $(1,1)-$form on $X$ with smooth local potentials such that its cohomology class $\eta := \{\omega\}$ is big, 
\item $\mu \geq 0$ is a continuous  volume form on $X$, 
  \item $\f : X  \longrightarrow \R$ is the unknown function,
   \end{itemize}

\begin{thm} \label{thm:big}
 Let $\f$ (resp. $\p$) be a  viscosity subsolution (resp. supersolution) to 
 (\ref{eq:CMA}) in $X $.
 Then  
$$
\f  \leq \p  \, \, \, \mathrm{in} \, \, \,  \mathrm{Amp} \{\omega \}.
$$
In particular if $\p$ is continuous  on $X$, then $\f \leq \psi$ on $X$.
\end{thm}

Recall that the ample locus $\mathrm{Amp} \{ \omega \}$ of the cohomology class of $\omega$ is the  Zariski open subset
of points $x \in X$ such that there exists a positive closed current cohomologous to $\omega$
which is a K\"ahler form near $x$. In particular $\mathrm{Amp}  \{ \omega \}=X$ when $\omega$ is K\"ahler.

\subsection{A refined local comparison principle}

We first establish a   useful  lemma for the local equation
\begin{equation}  \label{eq:LCMA}
e^{u } \mu - (dd^c u)^n = 0.
\end{equation}

\begin{lem} \label{Fund} 
Let $\mu (z) \geq  0, \nu (z)\geq 0$ be continuous volume forms on some domain $D \Subset \C^n$. 
 Let $u$ be a  subsolution to   (\ref{eq:LCMA}) associated to $\mu$  and let $v$ be a bounded  supersolution to 
   (\ref{eq:LCMA}) associated to $\nu$ in $ D$.
 Assume that 
 
 \smallskip
 
 $(i)$ the function   $u - v$ achieves a local  maximum at some $x_0 \in D$;
 
 \smallskip
  
 $ (ii)$ $\exists c > 0$ s.t. $z \longmapsto u (z)- 2 c \vert z\vert^2$ is plurisubharmonic near $x_0$.
 
\noindent  Then   $\nu(x_0)>0$ and 
 \begin{equation} 
 e^{ u (x_0)} \mu (x_0)  \leq e^{ v (x_0)} \nu (x_0).
\end{equation}
\end{lem}

 \begin{proof}
The idea of the proof  is to apply Jensen-Ishii's comparison principle which is obtained from the classical maximum principle 
by regularizing $u$ and $v$  (see \cite{CIL92,EGZ11}).

We can assume  that   $u - v$ achieves a local maximum  at $0$,
 $$
 M := \sup_{x \in  \bar \B} (u (x) - v  (x)) = u (0) - v  (0),
$$
where $\B$ is the unit ball in $\C^n$.
The hypothesis (ii) insures that $u (z) - 2 c \vert z\vert^2$ is  plurisubharmonic in a neighborhood of $\bar \B$. 
Thus  for any fixed $\alpha \in ]0,1[$, the function 
$$
  u_\alpha(z) := u (z) -  c \alpha \vert z\vert^2
$$
 is  strictly psh in $\B$  and $u_\alpha - v$ achieves a strict maximum at $0$ in $ \bar \B$.
 
 Observe that
 $
 u_\alpha - (1-\alpha) u  - \alpha c \vert z \vert^2 =    \alpha (u -  2 c \vert z \vert^2 )
 $
 is psh in $\B$ hence  
 $$
 dd^c u_\alpha \geq (1-\alpha) dd^c u + \alpha \, c \,  dd^c \vert z \vert^2.
 $$ 
It follows that, setting $dV_{eucl} :=  (dd^c \vert z\vert^2)^n$,
\begin{eqnarray*}
(dd^c u_\alpha)^n & \geq & (1-\alpha)^n (dd^c u)^n +\alpha^n c^n dV_{eucl} \\
& \geq & e^{ u_\alpha(z)  +  n \log (1-\alpha) } \mu + \alpha^n c^n dV_{eucl}
\end{eqnarray*}
in the viscosity sense, noticing that $u \geq u_\alpha$.

Since  $u_\alpha (0) = u (0)$ we replace in the sequel $u$ by $u_\alpha$.
We have thus reduced the situation to the case where $u (z)-  c \vert z\vert^2$ is  psh in $\B$, 
$u-v$ achieves a strict maximum in $ \bar \B$ at $ 0$ and $u$ is a subsolution of the equation  
$$
  e^{ u (z) +    n \log (1-\alpha)} \mu +  \alpha^n  c^n dV_{eucl}  - (dd^c u)^n = 0. 
$$

We want to apply  Jensen-Ishii's maximum principle to $u$ and $v$ by using the penalty method as in \cite{CIL92}.  
Fixing  $\e>0$, we want to maximize on $ \bar \B \times \bar \B$
the upper semi-continuous  function
 $$
  w_\e (x,y) :=   u (x) - v  (y)  - (1 \slash 2 \e) \vert x - y\vert^2. 
 $$
 
The  penalty function forces the maximum of $w_\e$ to be asymptotically attained  along the diagonal. 
 Since   $w_\e$ is upper semicontinuous  on the compact set $\bar \B \times \bar \B$, 
there exists $(x_{{\e}},y_{\e}) \in  \bar \B \times \bar \B$ such that
\begin{eqnarray*}
 M_{{\e}}& := &\sup_{(x,y) \in \bar \B^2} \left\{ u (x)- v (y) -\frac{1}{2 \e} \vert x - y\vert^2 \right\}\\
 &=&  u (x_{{\e}})- v  (y_{{\e}}) -\frac{1}{2\e} \vert x_{\e} - y_{\e}\vert^2.
\end{eqnarray*}
 
 The following result  is classical  \cite[Proposition 3.7]{CIL92}:
 
 \begin{lem} \label{asymp}
 We have $ \vert x_{\e} - y_{\e}\vert^2  = o ({\e})$. Every limit
 point $(\hat x, \hat y)$ of $(x_{{\e}}, y_{{\e}})$ satisfies $\hat x= \hat y$, $(\hat x,\hat x) \in  \bar \B \times \bar \B$ and 
 $$
 \lim_{\e \to 0} M_\e = \lim_{\e \to 0}  (u (x_{{\e}})- v  (y_{\e})) =   u (\hat x)- v (\hat x)= M.
 $$
 \end{lem}
 
Our hypothesis of strict local maximum guarantees that $\hat x = 0$.    
 Therefore $(x_{{\e}}, y_{{\e}}) \to (0,0)$ as $\e \to 0$. Take any sequence  $(x_j,y_j) = (x_{\e_j}, y_{\e_j}) \to (0,0) $ 
 and $(x_j,y_j)  \in \B^2$ for any $j > 1$ so that  the conditions in the lemma above are satisfied with $\e = \e_j \to 0$.

Set  $\phi_j (x,y) :=\frac{1}{2\e_j} \vert x - y \vert^2 $. The function $(x,y) \mapsto  u(x) - v(y) - \phi_j (x,y) $ achieves its maximum   in 
$ \bar \B^2$ at the interior point $(x_j,y_j)  \in  \B^2$.
We can thus apply Jensen-Ishii's  maximum principle  and obtain the following estimates:

\begin{lem}\label{main}
For any $\gamma > 0$, we can find $(p^+, Q^+), (p_-, Q_-)\in \C^n \times Sym_{\R}^2 (\C^n)$
such that
\begin{enumerate}
 \item $(p^+, Q^+)\in \overline{\mathcal J}^{2+}  u (x_j)$, $(p_-, Q_-)\in 
\overline{\mathcal J}^{2-} v (y_j)$, where 
$$
p^+ =  \frac{(x_j - y_j)}{2 \e_j} + p_-.
$$
\item The block diagonal matrix with entries $(Q^+, Q_-)$ satisfies:
$$
-(\gamma^{-1}+ \| A \| ) I \le 
\left(
\begin{array}{cc}
Q^+ & 0 \\
0 & -Q_-
\end{array}
\right)
\le A+\gamma A^2, 
$$
where $A=D^2\phi_j (x_{j}, y_{j}) =\gamma^{-1}
\left(
\begin{array}{cc}
I &  -I\\
-I &  I
\end{array}
\right),$
and $\| A \|$ is the spectral radius of $A$. 
\end{enumerate}
\end{lem}

We choose $\gamma = \e_j$. Thus
$$
-(2\e_j^{-1} ) I \le 
\left(
\begin{array}{cc}
Q^+ & 0 \\
0 & -Q_-
\end{array}
\right)
\le \frac{3}{\e_j} \left(
\begin{array}{cc}
I &  -I\\
-I &  I
\end{array}
\right).
$$

Looking at the upper and lower diagonal terms we deduce that the eigenvalues of $Q^+, Q_-$ are $O(\e_j^{-1})$. Evaluating the inequality on vectors of the form $(Z,Z)$ we deduce that $Q^+\leq Q_-$ in the sense of quadratic forms.

For a fixed $Q\in Sym_{\R}^2 (\C^n)$, denote by $H = Q^{1,1}$ its $(1,1)$-part. It is a hermitian matrix. 
Since $(p^+, Q^+)\in \overline{\mathcal J}^{2+} u (x_j)$, 
we deduce from the viscosity differential inequality satisfied by $u$ 
that $H^+$ is positive definite  (see \cite[Theorem 2.5]{EGZ11}).  The inequality $ H^+ \le H_-$
forces $H_->0$, thus $0 \leq H^+ \leq H_-$ and  
$ \mathrm{det} \, H^+\leq \mathrm{det} \, H_-$.

\smallskip

The viscosity differential inequalities satisfied by $ u$ and $v$ yield
$$
e^{ u (x_j) +   n \log (1 - \alpha)} \mu (x_j)  + \alpha^n c^n dV_{eucl}(x_j)
\leq  det H^+  \leq  det H_-  \leq e^{  v (y_j)} \nu (y_j) ,
$$
hence
\begin{equation} \label{eq:FI0}
e^{  u (x_j)  + n \log (1 - \alpha)} \mu (x_j)  + \alpha^n c^n dV_{eucl}(x_j)
  \leq  e^{  v (y_j)} \nu (y_j).
\end{equation}

Recall that
 $
 \lim_{j \to + \infty}  ( u (x_j)- v  (y_j)) =  u (0)- v (0).
$
On the other hand the semi-continuity properties of $u,v$ yield (taking a subsequence if necessary),
$
 \lim_{j \to + \infty}  u (x_j) \leq  u (0), \, \, \lim_{j \to + \infty}  v (y_j) \geq v (0).
$
Hence
$$
 \lim_{j \to + \infty}  u (x_j) = u (0), \, \, \lim_{j \to + \infty}  v (y_j) = v (0).
$$
 Letting $j \to \infty$ in the inequality (\ref{eq:FI0}) we infer  
$$ 
e^{   u(0) + n \log (1-\alpha)} \mu (0) + \alpha^n c^n dV_{eucl}(0) \leq e^{  v (0))} \nu (0).
$$
Since $dV_{eucl}(0)>0$, we infer $\nu(0)>0$. Letting $\alpha \rightarrow 0$ yields the desired inequality.
  \end{proof}

  \subsection{Proof of Theorem \ref{thm:big}}

We are now ready for the proof of Theorem \ref{thm:big}.

\begin{proof}
  Assume that $\f$  is a  subsolution to the  complex 
Monge-Amp\`ere equation (\ref{eq:CMA}) associated to $\mu$  and $\p$ is a   supersolution to the  complex Monge-Amp\`ere equation (\ref{eq:CMA}) associated to $\nu$  in $X$.

Fix $\e \in ]0,1[$ and set 
$$
\tilde \f (x) := (1 - \e) \f(x)  + \e \rho(x),
$$
where  $\rho \leq \f$ is a $\omega$-psh function satisfying $\omega + dd^c \rho \geq \beta$, where $\beta$ is a 
K\"ahler form on $X$. Such a function exists since the cohomology class $\eta$ of $\omega$ is big. 
One can moreover impose $\rho$ to be smooth in the ample locus $\Omega:=\mathrm{Amp} \{ \omega \}$,
with analytic singularities,
and such that $\rho(x) \rightarrow -\infty$ as $x \rightarrow \partial \Omega = X \setminus \Omega$.

Since $ \tilde \f - \p$ is bounded from above on $X$, tends to $- \infty$ when  $x \to \partial \Omega$,  
and is upper semicontinuous in $\Omega$,
the maximum of  $\tilde  \f - \p$ is achieved at some point $x_0 \in \Omega$,
$$
M := \sup_{x \in  X} (\tilde \f (x) - \p  (x)) = \tilde \f (x_0) - \p  (x_0).
$$

Observe that  $\tilde \f$ satisfies  
$
(\omega + dd^c \tilde \f)^n \geq (1 - \e)^n (\omega + dd^c \f)^n + \e^n \beta^n,
$
in the viscosity sense in $ \Omega$.
Now $\tilde \f \leq \f$ since $\rho < \f$ hence
$$
(\omega + dd^c \tilde \f)^n \geq e^{\tilde \f} \left\{ (1 - \e)^n \mu(x) + e^{-C}\e^n \beta^n \right\},
$$
where $\tilde{\f} \leq \f \leq C$.

The idea is to localize near $x_0$ and use Lemma~\ref{Fund}. 
Choose  complex coordinates  $z = (z^1, \ldots, z^n)$ near $x_0$
defining a biholomorphism identifying a closed neighborhood  of $x_0$ to 
the closed complex ball $\bar B_2 := B(0,2) \subset \C^n$ of radius $2$, sending $x_0$ to the origin in $\C^n$.

We let $h_{\omega}  (x)$  be a smooth local potential for $\omega$ in $ B_2$,
i.e. $dd^c h_{\omega} = \omega $ in $ B_2$.
Setting
$u := \tilde \f \circ z^{- 1} + h_{\omega} \circ z^{- 1}  $ in $ B_2$ we obtain
\begin{equation} \label{eq:subsol} 
e^{ u } \tilde \mu \leq 
(dd^c u)^n, \, \, 
\text{ in }  B_2,
\end{equation}
where 
$$
\tilde{\mu}=  e^{- h_{\omega} \circ z^{- 1}} z^{*} \left\{ (1 - \e)^n \mu(x) + e^{-C}\e^n \beta^n \right\} > 0
$$ 
is a  continuous volume form on $ B_2$.

Similarly  the lower semi-continuous function 
$v  := \p \circ z^{- 1}  + h_{\omega} \circ z^{- 1} $
 satisfies the viscosity differential inequality
\begin{equation} \label{eq:supersol}
e^{ v} \tilde \nu \geq (dd^c v)^n, \, \, 
\text{ in }  B_2,
\end{equation} 
where $\tilde \nu := e^{- h_{\omega} \circ z^{- 1}}  z^{*} (\nu) >  0$ 
is a  continuous volume form on $  B_2$.

Our hypothesis guarantees
\begin{equation} \label{eq:LocMax}
M  =  \sup_{ X} \{\tilde \f  - \p\} = \max_{\bar \B} \{ u (\zeta) -  v(\zeta)\} = u (0) - v (0),
\end{equation} 
i.e.   $ u  -  v$ achieves its maximum at the interior point $ 0 \in  \B$. Moreover   
$dd^c u = \omega + dd^c \f \geq \e \beta$, i.e.
$ u$ is $2 c$-strictly psh in $B_2$ for some $c=c(\e)>0$.
We  apply Lemma~\ref{Fund} and conclude that  $\tilde \nu (0)>0$ and
 \begin{equation*}
 e^{    u (0)}  \tilde \mu (0)    \leq e^{ v (0)} \tilde \nu (0),
\end{equation*} 

Going back to $\f$ and $\p$ we obtain  $\nu(x_0)>0$ and
$$
(1-\e)^n e^{\tilde \f(x_0)} \mu(x_0) \leq e^{\p(x_0)} \nu(x_0).
$$

When $\mu = \nu$ we can divide by $\nu(x_0)=\mu(x_0)>0$ and obtain
$$
(1-\e) \f(x)+\e \rho(x) \leq \p(x)-n \log(1-\e),
$$
for all $x \in X$ and $0 < \e <1$.
Letting $\e \to 0$,  we infer
$
\f  \leq \psi, 
$
 in $X \setminus  \{\rho = - \infty\} = \Omega$. The set  $\{\rho = - \infty\}$ has Lebesgue measure $0$ 
 hence  the inequality $\f  \leq \psi$  holds on $X$ if $\psi$ is   continuous on $X$. 
\end{proof}

  \section{Further modifications}

  \subsection{Statements of \cite{EGZ11}}
  
  The definitions and statements have to be modified as follows when working on compact complex manifolds:
  \begin{itemize}
  \item  a viscosity subsolution is bounded from above on $X$, 
  u.s.c. in    $\rm{Amp}\{\omega\}$  where it satisfies the corresponding differential inequalities ;
  \item a viscosity supersolution is bounded from below,
  l.s.c.   in    $\rm{Amp}\{\omega\}$ where it satisfies the corresponding differential inequalities;
  \item a viscosity solution is both a subsolution and a supersolution, in particular it is bounded on $X$
  and continuous in     $\rm{Amp}(\{\omega\})$.
  \end{itemize}

  To construct the unique viscosity solution we proceed as previously done, using the Perron method: 
  the family of subsolution is non empty, it is uniformly bounded from above by a {\it continuous}
  supersolution (e.g. a constant).
  
  Since   the comparison principle is only shown to hold in the ample locus,   the solutions
  we construct are continuous in the ample locus rather than in all of $X$.
  Thus the statements of Theorem A, Corollary B, Theorem C (see also Corollary 3.4, Corollary 3.5,
  Theorem 3.6 and Corollary 3.7)
  have to be modified accordingly, replacing "continuous" by
  "continuous in the ample locus" or "bounded on $X$ and continuous in the ample locus", etc.

    \subsection{Continuous approximation of quasi-psh functions}

  The proof of the main result of \cite{EGZ15} has to be modified similarly.
  It provides an approximation process by $\omega$-psh functions which are merely continuous 
  in the ample locus.
    It is an interesting open problem to decide whether 
 these approximants are actually globally continuous on $X$  (see \cite[Definition 2.2]{EGZ09}).
    It follows from \cite{CGZ13} that this is the case when $\omega=\pi^* \omega_Y$ is the pull-back
    of  a Hodge form on a singular projective variety $Y$, under a desingularization $\pi: X \rightarrow Y$.

  \subsection{Parabolic theory}

A similar problem occurs in the localization technique used in \cite{EGZ16} to prove the  parabolic viscosity comparison principle.
The method proposed in this note can be adapted to the parabolic setting and yields an alternative proof 
of the parabolic comparison principle valid in the ample locus.
We will give the details elsewhere.

\end{document}